\pdfoutput=1
\documentclass[hidelinks, reqno]{amsart}
\usepackage[english,french]{babel}
\usepackage{graphicx}
\usepackage{subcaption}
\usepackage{tabularx}
\usepackage{booktabs}
\usepackage{float}
\usepackage{array}
\usepackage{amsmath}
\usepackage{amsfonts}
\usepackage{amssymb}
\usepackage{amsthm}
\usepackage[scr]{rsfso}
\usepackage{enumitem}
\usepackage{permute}
\usepackage{slashed}
\usepackage[usenames,dvipsnames]{xcolor}
\usepackage[pagebackref = true, colorlinks, linkcolor = Red, citecolor = Green, bookmarksdepth=2, linktocpage=true]{hyperref}
\usepackage[capitalize]{cleveref}
\usepackage{comment}
\usepackage{changepage}
\usepackage{dsfont}

\usepackage[T1]{fontenc}
\usepackage{makecell}

\usepackage{tikz}
\usetikzlibrary{calc}

\allowdisplaybreaks

\DeclareMathOperator{\tr}{tr}

\DeclareMathOperator{\vol}{vol}

\DeclareMathOperator{\diam}{diam}

\DeclareMathOperator{\ad}{ad}

\DeclareMathOperator{\p}{\partial}
\DeclareMathOperator{\area}{area}

\newcommand{\R}{\mathbb{R}}

\newcommand{\LieG}{\mathfrak{g}}
\newcommand{\LieA}{\mathfrak{a}}

\newcommand{\LieK}{\mathfrak{k}}

\newcommand{\LieP}{\mathfrak{p}}

\makeatletter
\newcommand{\mytag}[2]{%
\text{#1}%
\@bsphack
\begingroup
\@onelevel@sanitize\@currentlabelname
\edef\@currentlabelname{%
\expandafter\strip@period\@currentlabelname\relax.\relax\@@@%
}%
\protected@write\@auxout{}{%
\string\newlabel{#2}{%
{#1}%
{\thepage}%
{\@currentlabelname}%
{\@currentHref}{}%
}%
}%
\endgroup
\@esphack
}
\makeatother

\DeclareFontFamily{U}{mathb}{\hyphenchar\font45}
\DeclareFontShape{U}{mathb}{m}{n}{
<5> <6> <7> <8> <9> <10> gen * mathb
<10.95> mathb10 <12> <14.4> <17.28> <20.74> <24.88> mathb12
}{}
\DeclareSymbolFont{mathb}{U}{mathb}{m}{n}
\DeclareMathSymbol{\bigast}{2}{mathb}{"06}

\def\XXint#1#2#3{{\setbox0=\hbox{$#1{#2#3}{\int}$}
\vcenter{\hbox{$#2#3$}}\kern-.5\wd0}}

\theoremstyle{plain}
\newtheorem{theorem}{Theorem}[section]
\newtheorem{proposition}[theorem]{Proposition}
\newtheorem{lemma}[theorem]{Lemma}
\newtheorem{corollary}[theorem]{Corollary}
\newtheorem{conjecture}[theorem]{Conjecture}

\theoremstyle{definition}
\newtheorem{definition}[theorem]{Definition}

\theoremstyle{remark}
\newtheorem{remark}[theorem]{Remark}

\Crefname{enumi}{Property}{Properties}
\Crefname{alternativei}{Alternative}{Alternatives}
\Crefname{subsection}{Subsection}{Subsections}

\numberwithin{equation}{section}

\begin{document}
\selectlanguage{english}


\title[Total curvature]{A total curvature estimate of closed hypersurfaces in non-positively curved symmetric spaces}

\author{Jiangtao Li}
\address{Department of Mathematics, UC San Diego, 9500 Gilman Drive, La Jolla, CA 92093-0112}
\email{jil320@ucsd.edu}

\author{Zuo Lin}
\address{Department of Mathematics, UC San Diego, 9500 Gilman Drive, La Jolla, CA 92093-0112}
\email{zul003@ucsd.edu}

\author{Liang Xu}
\address{Department of Mathematics, UC San Diego, 9500 Gilman Drive, La Jolla, CA 92093-0112}
\email{liangxu@ucsd.edu}

\date{\today}

\begin{abstract}
    In this paper, we prove a total curvature estimate of closed hypersurfaces in simply-connected non-positively curved symmetric spaces, and as a corollary, we obtain an isoperimetric inequality for such manifolds. 
\end{abstract}

\maketitle
\setcounter{tocdepth}{1}

\section{Introduction}

The Euclidean isoperimetric inequality asserts that for any compact domain $E\subseteq\R^{n+1}$ with smooth boundary $\p E$, there holds
\begin{equation}\label{eq:EuclideanIsoperimetric}
    \frac{\area(\p E)^{n+1}}{\vol(E)^n}\ge \frac{\area(\mathbb{S}^n)^{n+1}}{\vol(B^{n+1})^n}
\end{equation}
where $B^{n+1}$ is the unit $(n+1)$-ball. The equality holds if and only if $E$ is an $(n+1)$-ball. 

The following conjecture, which generalizes this classical result, appeared in \cite{Aub76,BZ13,Gro81}:

\begin{conjecture}[Cartan-Hadamard]\label{conj:CartanHadamard}
    \cref{eq:EuclideanIsoperimetric} is true for compact domain $E$ with smooth boundary $\p E$ in any Hadamard manifolds $N^{n+1}$, i.e. non-positively curved simply-connected manifolds.
\end{conjecture}
 This conjecture was settled when the ambient dimension $n+1=2,3,4$ (cf. \cite{Wei26,Kle92,Cro80}). Some progress towards higher dimensional cases has been made over the past few decades (cf. \cite{Cro80,MJ00,GS22}).

Intuitively, the gradient of the area functional is measured by the mean curvature. Note that the mean curvature can be controlled via \emph{the total curvature estimate}
\begin{equation}\label{eq:TotalCurvature}
    \int_{\p E}|GK|d\area\ge\area(\mathbb{S}^n), 
\end{equation}
where $GK$ is the Gauss-Kronecker curvature of $\p E$. In their recent work \cite{GS22}, Ghomi and Spruck proved that the total curvature estimate implies \cref{conj:CartanHadamard} in all dimensions, which provides an important motivation to study \cref{eq:TotalCurvature}. 

Using the method of Chern-Lasholf \cite{CL57}, \cref{eq:TotalCurvature} was established in hyperbolic spaces of all dimensions in \cite{Bor02}. In \cite{Bor01}, \cref{eq:TotalCurvature} was also shown to be true in a special class of Hadamard manifolds equipped with conformally flat metrics. Except in special cases, the validity of the total curvature estimate is still unknown in general. 

Generalizing the idea of \cite{Bor01}, we prove the following result:

\begin{theorem}\label{thm:Main}
     Let $N^{n + 1}$ be a simply-connected non-positively curved symmetric space with sectional curvature bounded below by $-\kappa^2$ and $M \subseteq N$ be a closed embedded hypersurface whose diameter in $N$ is bounded above by $D$. Then the following total curvature estimate holds:
    \begin{align}\label{eq:TotalCurvatureMain}
        \int_M |GK| \, d\area_M \geq e^{-n(n+1)\kappa D}\area(\mathbb{S}^n).
    \end{align}
\end{theorem}
An immediate consequence is a Willmore-type inequality: 
\begin{equation}
\int_M \left| H/n \right|^n d\area_M \geq e^{-n(n+1)\kappa D} \area(\mathbb S^n), 
\end{equation}
where $H$ stands for the mean curvature of $M$. As a corollary, we also obtain the following isoperimetric inequality.

\begin{corollary}\label{cor:IsoperimetricIneq}
    Let $N^{n+1}$ be a simply-connected non-positively curved symmetric space, with sectional curvature bounded below by $-\kappa^2$, and of dimension $n+1\leq 7$. Then for any compact domain $E\subseteq N$ with smooth boundary $\p E$ and $\diam(E)\leq D$, there holds
    \begin{equation}
        \frac{\area(\partial E)^{n+1}}{\vol(E)^n} \geq e^{-2n(n+1)\kappa D}\frac{\area(\mathbb{S}^n)^{n+1}}{\vol(B^{n+1})^n}. 
    \end{equation}
\end{corollary}
From this it is then standard to deduce the Sobolev inequality for $f\in W^{1,1}(E)$:
\begin{equation}
    (n+1)(\vol (B^{n+1}))^\frac{1}{n+1}e^{-2n\kappa D} \left(\int_E |f|^\frac{n+1}{n} d\vol \right)^\frac{n}{n+1} \leq \int_E|\nabla f|d\vol + \int_{\p E}|f|d\area. 
\end{equation}

\subsection*{Organization of the Paper}
The proof of \cref{thm:Main} goes as follows. In \cref{sec:SymmetricSpace} and \cref{sec:Busemann}, we introduce Busemann functions on symmetric spaces and study its properties. For any closed embedded hypersurface $M^n\subseteq N^{n+1}$, in \cref{sec:GaussMap} we construct the generalized Gauss map of $M$ and discuss its regularity using Busemann functions. In \cref{sec:Jacobi}, it is proved that the Jacobian of the generalized Gauss map is bounded from above by $|GK|$ over the closed contact subset $F \subseteq M$. Finally, in \cref{sec:ProofMain}, it is shown that \cref{eq:TotalCurvatureMain} follows by applying the area formula to the generalized Gauss map restricted to $F$.

\subsection*{Notations} In this paper, we will use $T^1N$ to denote the unit sphere bundle in the tangent bundle of the Riemannian manifold $N$. For a $C^2$ function $f$ on $N$, we take $\nabla^2f$ as a $(1,1)$ tensor on $N$. If $M\subseteq N$ is a hypersurface with unit normal vector field $\nu$, we define the associated shape operator by $A(X)=\nabla_X\nu$. The Gauss-Kronecker curvature $GK(x)$ at $x\in M$ is defined as $\det A(x)$. For a linear operator $L$ on an inner product space $V$, $\|L\|_{\mathrm{op}}=\sup_{|v|=1}|Lv|$.

\section{Simply-connected non-positively curved symmetric spaces}\label{sec:SymmetricSpace}

We fix some notions related to symmetric spaces. For a comprehensive treatment of Lie groups and symmetric spaces, we refer to \cite{Kna02} and \cite{Hel78}. 

Let $(N, g)$ be a simply-connected non-positively curved symmetric space. By \cite[Chapter V, Proposition 4.2; Chapter VIII, Proposition 5.5]{Hel78}, we have the decomposition of $N$ as a Riemannian manifold
\begin{align}\label{eqn:DecompositionSymmetricHadamard}
    (N, g) \cong (\R^\ell, g_0) \times \prod_{i = 1}^{k} (N_i, g_i),
\end{align}
where $g_0$ is the standard Euclidean metric and $(N_i, g_i)$'s are irreducible non-Euclidean non-positively curved symmetric spaces. 

We now set $(N, g)$ be an \emph{irreducible} non-Euclidean non-positively curved symmetric space and $G$ be the identity component of the isometry group of $N$. For every point $x \in N$, we write $s_x$ and $K_x$ be the corresponding inversion symmetry at $x$ and the isotropic group at $x$. Let $\theta_x$ be the corresponding Cartan involution to $s_x$. We have the following Cartan decomposition for the Lie algebra $\LieG$ of $G$:
\begin{align*}
    \LieG = \LieK_x \oplus \LieP_x
\end{align*}
with $\LieK_x = \mathrm{Lie}(K_x)$ and $\LieP_x$ corresponding to eigenvalues $+1$ and $-1$ of $\theta_x$. We can naturally identify $\LieP_x$ with $T_x N$. 

Let $\beta(X, Y) = \tr(\ad X \circ \ad Y)$ be the Killing form on $\LieG$. We have the following standard $K_x$-invariant positive-definite inner product on $\LieG$: 
\begin{align*}
    \beta_{\theta_x}(v, w) := - \beta(v, \theta_x(w)).
\end{align*}
Restricting to $\LieP_x \cong T_xN$, it gives a $G$-invariant Riemannian metric on $N$ and it equals $\beta$. The irreducibility of $N$ implies that $g=\lambda\beta$ where $\lambda>0$ is a positive constant. 

Let $o \in N$ be a fixed base point. We will drop the subscript for every thing based at $o$. Let $\LieA \subseteq \LieP$ be a maximal abelian subspace. We have the following restricted-root space decomposition (cf. \cite[Proposition 6.40]{Kna02}):
\begin{align*}
    \LieG = \mathfrak{Z}(\LieA) \oplus \bigoplus_{\alpha \in \Phi} \LieG_{\alpha},
\end{align*}
where $\mathfrak{Z}(\LieA)$ is the centralizer of $\LieA$ in $\LieG$ and $\LieG_{\alpha}$ is the nonzero eigenspace corresponding to the restricted root $\alpha \in \LieA^*$. 

We relate $\lambda$ with lower bound of sectional curvature of $N$ in the following lemma. 

\begin{lemma}\label{lem:CurvatureRescale}
    Suppose $(N, g)$ is an irreducible non-Euclidean non-positively curved symmetric space with the sectional curvature $K_N \geq -\kappa^2$, we have
    \begin{align*}
        \lambda \geq \frac{\max_{\alpha \in \Phi}|\alpha|^2}{\kappa^2},
    \end{align*}
    where $|\alpha|$ is the operator norm by viewing $\alpha$ as a linear functional on $\LieA$. 
\end{lemma}

\begin{proof}
    By \cite[P137, Example 2]{Ebe96}, $(N, \beta)$ admits a totally geodesic $2$-plane with sectional curvature $-|\alpha|^2$ for every $\alpha \in \Phi$. Then we have $\mathrm{sec}(N, \beta) \leq - \max_{\alpha \in \Phi} |\alpha|^2$. Rescaling to $(N, g)$, we have
    \begin{align*}
        \frac{-\max_{\alpha \in \Phi} |\alpha|^2}{\lambda} \geq -\kappa^2,
    \end{align*}
    which implies the lemma. 
\end{proof}

\section{Busemann functions}\label{sec:Busemann}
We briefly review properties of Busemann functions needed in this paper. We refer to \cite{BH99, BGS85, Ebe96} for a comprehensive treatment. 

\subsection{Busemann functions on Hadamard manifolds}

Let $N$ be a Hadamard manifold and let $o \in N$ be a base point. We define the Busemann function based at $o$ as the following. 

Let $v \in T_o^1N$ be a unit tangent vector at $o$. There is a geodesic $\gamma_v:\R \to N$ with $\gamma_v(0) = o$ and $(\gamma_v)'(0) = v$. The Busemann function $B_v$ is defined to be
\begin{align*}
    B_v(x) = \lim_{t \to +\infty} \bigl(d(x, \gamma_v(t)) - t\bigr).
\end{align*}
The limit always exists (cf. \cite[Lemma 8.18]{BH99}). 

We record the properties of Busemann functions on Hadamard manifolds needed in this paper in the following proposition. 
\begin{proposition}\label{pro:RegularityBusemann}
    The Busemann functions on Hadamard manifolds satisfy the following properties. 
    \begin{enumerate}
        \item The function $(v, x) \mapsto B_v(x)$ is continuous. 
        \item Fix $v$, the function $x \mapsto B_v(x)$ is a $C^2$ convex function. 
        \item Let $(N_1, g_1)$ and $(N_2, g_2)$ be two Hadamard manifolds. For every unit tangent vector in $T_{(o_1, o_2)}^1 (N_1 \times N_2)$, it has the form $(\cos(\theta))v + (\sin(\theta))w$ where $v \in T_{o_1}^1 N_1$ and $w \in T_{o_2}^1 N_2$ and $\theta \in [0, \frac{\pi}{2}]$. The corresponding Busemann function has the following form:
        \begin{align*}
            B_{\cos(\theta)v + \sin(\theta)w}((x_1, x_2)) = \cos(\theta)B_v(x_1) + \sin(\theta)B_w(x_2).
        \end{align*}
    \end{enumerate}
\end{proposition}

\begin{proof}
    Property~(1) follows from \cite[Lemma 8.18(2)]{BH99}. 
    Property~(2) is \cite[Proposition 3.1]{HIH77} and \cite[Proposition 8.22]{BH99}. Property~(3) follows from \cite[Example 8.24(3)]{BH99}. 
\end{proof}

\subsection{Busemann functions on non-positively curved symmetric spaces}
We now prove some estimates of Hessian of Busemann functions on simply-connected non-positively curved symmetric spaces. The following is the main lemma in this subsection. 

\begin{lemma}\label{lem:HessianEstimateGeneral}
    Let $(N, g)$ be a simply-connected non-positively curved symmetric space with sectional curvature bounded below $-\kappa^2$. We have the following estimates on Hessian of Busemann functions and their differences. 

    \begin{enumerate}
        \item Let $v \in T_o^1 N$, we have that for all $x \in N$, 
    \begin{align}\label{eqn:HessianEstimateGeneral}
        \|\nabla^2 B_v(x)\|_{\mathrm{op}} \leq \kappa.
    \end{align}
        \item For any pair of unit tangent vectors $v,v'\in T_o^1 N$, we have that at $x$, 
    \begin{align}\label{eqn:HessianDifferenceEstimateGeneral}
        \|\nabla^2 B_v(x)-\nabla^2 B_{v'}(x)\|_{\mathrm{op}} \leq \kappa (\dim N)  |\nabla B_v(x) - \nabla B_{v'}(x)|.
    \end{align}
    \end{enumerate}   
\end{lemma}

The rest of this subsection is devoted to the proof of \cref{lem:HessianEstimateGeneral}. Proof of property~(1) follows exactly from the same lines as property~(2), so we will just present the proof of property~(2). For simplicity, we set $u = \nabla B_v(x)$ and $u' = \nabla B_{v'}(x)$.

Since Busemann functions on Euclidean spaces are linear functionals, the lemma automatically holds on Euclidean spaces. 

Now we deal with the case when $(N, g)$ is an \emph{irreducible non-Euclidean} non-positively curved symmetric space. Recall from \cref{sec:SymmetricSpace} that $g = \lambda \beta$ where $\beta$ is the Killing form associated to $G$. 

If $\lambda = 1$, we have the following calculation for Hessian of Busemann function on symmetric spaces from \cite[Lemma 4.9]{Dav23}. 
\begin{lemma}\label{lem:HessianBusemann}
    Let $v\in T^1_o N$ and $x \in N$. Recall $u = \nabla B_v(x)$. We have
    \begin{equation}\label{hessian}
        \nabla^2B_{v}(x)=\biggl(\sqrt{\ad_u^2}\biggr)\biggl|_{\LieP_x}.
    \end{equation}
    On the right hand side we are taking the unique square root of the self-adjoint semi-positive operator $\ad_u^2$ on $\LieG$ and then restrict it to the subspace $\LieP_x$. 
\end{lemma}

Therefore, 
    \begin{align*}
        \|\nabla^2B_v(x)-\nabla^2B_{v'}(x)\|_{\mathrm{op}} = \Bigg\|\biggl(\sqrt{\ad_u^2}\biggr)\biggl|_{\LieP_x} -\biggl(\sqrt{\ad_{u'}^2}\biggr)\biggl|_{\LieP_x}\Bigg\|_{\mathrm{op}} \leq \Bigg\|\sqrt{\ad_u^2}-\sqrt{\ad_{u'}^2}\Bigg\|_{\mathrm{op}},
    \end{align*}
where on the rightmost both operators are considered as operators acting on $\LieG$. 

By \cite[Lemma 5.2]{Ara71} and singular-value decomposition, we have
\begin{align*}
    \biggl\|\sqrt{\ad_u^2}-\sqrt{\ad_{u'}^2}\biggr\|_{\mathrm{op}} \leq \sqrt{\dim \LieG} \|\ad_u - \ad_{u'}\|_{\mathrm{op}} \leq \sqrt{\dim \LieG}\max_{\alpha \in \Phi}|\alpha||u - u'|.
\end{align*}
In conclusion, when $\lambda = 1$, 
\begin{align*}
    \|\nabla^2B_v(x)-\nabla^2B_{v'}(x)\|_{\mathrm{op}} \leq \sqrt{\dim \LieG}\max_{\alpha \in \Phi}|\alpha||u - u'|.
\end{align*}
For general $\lambda > 0$, a simple rescaling shows
\begin{align*}
    \|\nabla^2B_v(x)-\nabla^2B_{v'}(x)\|_{\mathrm{op}} \leq{}& \frac{\sqrt{\dim \LieG}\max_{\alpha \in \Phi}|\alpha|}{\sqrt{\lambda}} |u - u'| \\
    \leq{}& \kappa \dim N |u - u'|\\
    ={}&\kappa \dim N |\nabla B_v(x) - \nabla B_{v'}(x)|.
\end{align*}
The last inequality follows from \cref{lem:CurvatureRescale} and the fact that $\dim \LieG \leq \dim N + (\dim N)(\dim N - 1)/2$. 

For general case, it suffices to show that if $(N, g) = (N_1, g_1) \times (N_2, g_2)$ and \cref{eqn:HessianDifferenceEstimateGeneral} holds for both $(N_1, g_1)$ and $(N_2, g_2)$, it holds for $(N, g)$. 

Indeed, property~(3) of \cref{pro:RegularityBusemann} implies
\begin{align*}
    \nabla^2B_{(\cos \theta)v_1 + (\sin \theta)v_2} =  \begin{pmatrix}
        (\cos \theta)\nabla^2B_{v_1} & \\
         & (\sin \theta)\nabla^2B_{v_2}
    \end{pmatrix}.
\end{align*}
The rest of \cref{lem:HessianEstimateGeneral} follows from Cauchy-Schwartz inequality. 

\section{Generalized Gauss map and its regularity}\label{sec:GaussMap}

Let $N^{n+1}$ be a Hadamard manifold and $M^n\subseteq N^{n+1}$ be a compact embedded hypersurface with outward unit normal vector field $\nu$. Fix a point $o$ in the bounded open subset whose boundary is $M$. We will construct a Gauss map of $M^n$ and show that it is Lipschitz continuous when $N^{n+1}$ is a non-positively curved symmetric space.

Firstly, for each unit tangent vector $u\in T^1_xN$, there exists a unique unit tangent vector $v\in T^1_oN$ such that $\nabla B_v(x)=u$ (cf. \cite[Proposition 8.2]{BH99}). We define the map $G^x_o:T^1_xN\to T^1_oN$ which maps $u$ to $v$. The effect of $G^x_o$ is translating all the fibers of $T^1N$ to a single fiber $T^1_oN$ via Busemann functions.

\begin{definition}[Gauss map]\label{def:GaussMap}
    The Gauss map of $M$ with respect to $\nu$ and $o$ is defined as $S_M:M\to T^1_oN\cong \mathbb{S}^n, S_M(x)=G^x_o(\nu(x))$. 
\end{definition}

For a general ambient space $N$, the regularity of $S_M$ is unknown. However, as we will show in \cref{thm:LipschitzGaussMap}, $S_M$ is a Lipschitz map if $N$ is a non-positively curved symmetric space. 

\begin{lemma}\label{lem:LipschitzGxo}
    Suppose $N$ is a simply-connected non-positively curved symmetric space with $K_N\ge -\kappa^2$ and $\diam(M)\leq D$. For any $x\in M$, $G^x_o:T^1_xN\to T^1_oN$ is a Lipschitz map with a uniform Lipschitz constant $e^{(n+1)\kappa D}$.
\end{lemma}
\begin{proof}
    Let $u,u'\in T^1_x N$ be a pair of unit tangent vectors at $x$. For simplicity, set $v=G^x_o(u), v'=G^x_o(u')$. Then, by definition, we have $\nabla B_{v}(x)=u$ and $\nabla B_{v'}(x)=u'$. We need to estimate $|v-v'|=|\nabla B_{v}(o)-\nabla B_{v'}(o)|$ by $|u-u'|=|\nabla B_{v}(x)-\nabla B_{v'}(x)|$. Let $\gamma$ be the unit speed geodesic in $N$ from $o$ to $x$. Define $f(t)=|\nabla B_{v}(\gamma(t))-\nabla B_{v'}(\gamma(t))|^2$. Then $f'(t)$ can be estimated as follows:
     \begin{align*}
         f'(t) &{}= 2g(\nabla_{\gamma'}\nabla (B_v-B_{v'}), \nabla(B_v- 
              B_{v'}))\\
              &{}=2g(\nabla^2(B_v-B_{v'})(\nabla(B_v-B_{v'})),\gamma')\\
              &{}\ge -2\|\nabla^2(B_v-B_{v'})\|_{\mathrm{op}}\cdot|\nabla(B_v-B_{v'})|\cdot|\gamma'|\\
              &{}\ge -2(n+1)\kappa |\nabla(B_v-B_{v'})|^2\\
              &{}=-2(n+1)\kappa f(t)
     \end{align*}
     where the last inequality follows from \cref{lem:HessianEstimateGeneral} applying at point $\gamma(t)$.
     Then the standard ODE theory implies 
     \begin{align*}
         f(t)\ge e^{-2(n+1)\kappa t}|v-v'|^2.
     \end{align*}
     In particular, when $t=d_N(x,o)$, this gives
     \begin{align*}
         |u-u'|\ge e^{-(n+1)\kappa d_N(x,o)}|v-v'|\ge e^{-(n+1)\kappa D}|v-v'|.
     \end{align*}
     This finishes the proof.
\end{proof}

Given the above lemma, we can prove the Lipschitz continuity of $S_M$.

\begin{theorem}[Lipschitz continuity]\label{thm:LipschitzGaussMap}
    Under the assumptions of \cref{lem:LipschitzGxo}, the Gauss map $S_M$ is Lipschitz continuous.
\end{theorem}
\begin{proof}
    Let $x\in M$ be any point and $\eta$ be a geodesic of $M$ such that $\eta(0)=x$. Then we have 
    \begin{align*}
         &{} \limsup_{t\to 0}\frac{1}{|t|}\left|S_M(\eta(t))-S_M(x)\right|\\
        &{}= \limsup_{t\to 0}\frac{1}{|t|}\left|G^{\eta(t)}_o(\nu(\eta(t)))-G^{\eta(t)}_o(\nabla B_{S_M(x)}(\eta(t)))\right|\\
        &{}\leq e^{(n+1)\kappa D}\lim_{t\to 0}\frac{1}{|t|}\left|\nu(\eta(t))-\nabla B_{S_M(x)}(\eta(t))\right|\\
        &{}\leq e^{(n+1)\kappa D}\left|\nabla_{\eta'(0)}\nu-\nabla_{\eta'(0)}\nabla B_{S_M(x)}\right|\\
        &{}= e^{(n+1)\kappa D}\left|A(\eta'(0))-\nabla^2B_{S_M(x)}(\eta'(0))\right|\\
        &{}\leq e^{(n+1)\kappa D}(\sup_M|A| + \kappa)\\
        &{}=: L(M,N).
    \end{align*}
    The first inequality is a consequence of \cref{lem:LipschitzGxo} and the last inequality follows from \cref{lem:HessianEstimateGeneral}. 
    
    The above calculation shows that there is $\epsilon_x>0$ depending on $x$, such that 
    \[|S_M(\eta(t))-S_M(x)|<(L(M,N)+1)|t|, \forall t\in (-\epsilon_x,\epsilon_x).\]

    Now let $x,y\in M$ be any pair of points and let $\eta$ be a unit speed minimizing geodesic of $M$ such that $\eta(0)=x,\eta(d)=y$, where $d=d_M(x,y)$. From above we know that for any $t\in[0,d]$, there is $\epsilon_t$, such that $|S_M(\eta(\tau))-S(\eta(t))|<(L(M,N)+1)|\tau-t|, \forall \tau\in(t-\epsilon_t,t+\epsilon_t)$. Therefore, by compactness, there exist finitely many times $0=t_0<t_1<\cdots<t_l=d$ such that $|S_M(\eta(t_{i+1}))-S_M(\eta(t_i))|<(L(M,N)+1)(t_{i+1}-t_i)$. Hence,
    \begin{align*}
        |S_M(y)-S_M(x)|&\leq\sum_{i=0}^l|S_M(\eta(t_{i+1}))-S_M(\eta(t_i))|\\
                   &\leq\sum_{i=0}^l (L(M,N)+1)(t_{i+1}-t_i)\\
                   &=(L(M,N)+1)d_M(x,y).
    \end{align*}
    Namely, $S_M$ is a Lipschitz map.
\end{proof}

\section{An estimate on the Jacobian}\label{sec:Jacobi}

We assume that $N^{n+1}$ is a simply-connected non-positively curved symmetric space with $K_N\ge -\kappa^2$ and $M^n\subseteq N^{n+1}$ is a closed hypersurface with $\diam(M)\leq D$ through this section.

The Lipschitz continuity proved in \cref{thm:LipschitzGaussMap} guarantees the existence of the differential $dS_M$ almost everywhere, thanks to the Rademacher's theorem (cf. \cite[Theorem 5.2]{Sim83}). We show that the Jacobian is controlled by the Gauss-Kronecker curvature $GK$ of $M$ on a subset $F$.

For each $v\in T^1_oN$, we consider the level hypersurfaces of the associated Busemann function $B_v$. Since $B_v$ is at least $C^2$ and $|\nabla B_v|\equiv 1$, this is a one-parameter family of $C^2$ hypersurfaces which foliates $N$. We consider those first contact points when the level decreases. More precisely, define $c_v=\inf\{c\in\mathbb{R}:B_v^{-1}(s)\cap M=\emptyset, \forall s>c\}<\infty$ and $F_v=B^{-1}_v(c_v)\cap M\ne\emptyset$. Set 
\begin{equation}
    F=\cup_{v\in T^1_oN}F_v.
\end{equation}
Then $F$ consists of first contact points in all directions.

\begin{remark}\label{rmk:PropertiesSetF}
    By the construction of $F$ it is straightforward that $S_M$ maps $F$ onto $T^1_oN\cong\mathbb{S}^n$. In addition, it is not hard to prove that $F$ is a closed subset of $M$ from \cref{pro:RegularityBusemann} property~(1). 
\end{remark}
The following lemma gives an estimate of $dS_M$ on $F$.
\begin{lemma}\label{lem:EstimateDifferential}
    Let $x\in F$ and suppose $S_M$ is differentiable at $x$. For any tangent vector $w\in T_xM$, the following inequality holds:
    \begin{equation}
        |dS_M(w)|\leq e^{(n+1)\kappa D}|A(w)-(\nabla^2 B_{S_M(x)})(w)|.
    \end{equation}
\end{lemma}
\begin{proof}
    The proof follows from the same calculation as we did in the proof of \cref{thm:LipschitzGaussMap}. Let $\eta(t)$ be a curve on $M$ such that $\eta(0)=x$ and $\eta'(0)=w$. Then 
    \begin{align*}
        |dS_M(w)|&=\lim_{t\to 0}\frac{1}{|t|}\left|S_M(\nu(\eta(t)))-S_M(x)\right|\\
               &=\lim_{t\to 0}\frac{1}{|t|}\left|G^{\eta(t)}_o(\nu(\eta(t)))-G^{\eta(t)}_o(\nabla B_{S_M(x)}(\eta(t)))\right|\\
               &\leq e^{(n+1)\kappa D}\lim_{t\to 0}\frac{1}{|t|}\left|\nu(\eta(t))-\nabla B_{S_M(x)}(\eta(t))\right|\\
               &=e^{(n+1)\kappa D}|A(w)-\nabla^2B_{S_M(x)}(w)|. \qedhere
    \end{align*}
\end{proof}    

\begin{lemma}\label{lem:CompareDeterminants}The following algebraic facts hold.
    \begin{enumerate}
        \item  If $M$ and $N$ are two $n\times n$ matrices and 
            \begin{align*}
                |Mv|\leq|Nv|, \forall v\in\mathbb{R}^n,
            \end{align*}
        then $|\det M|\leq|\det N|$.
        \item If $M$ and $N$ are semi-positive definite and $M\ge N\ge 0$, then $\det M\ge \det N$.
    \end{enumerate}
   
\end{lemma}
\begin{proof}
    Property~(1) follows from the singular-value decomposition. Property~(2) is Hadamard inequality. 
\end{proof}

The following Jacobian estimate is an immediate consequence of \cref{lem:EstimateDifferential} and \cref{lem:CompareDeterminants}.

\begin{corollary}[Jacobian estimate]\label{cor:JacobianEstimates}
    If $x\in F$ and $S_M$ is differentiable at $x$, then
    \begin{equation}
        J_{S_M}(x)=|\det(dS_M)_x|\leq e^{n(n+1)\kappa D}|\det A|=e^{n(n+1)\kappa D}|GK(x)|.
    \end{equation}
\end{corollary}
\begin{proof}
     Suppose $x\in F_v$. Then by definition of $F_v$, $\nabla B_v(x)=\nu(x)$ and thus $S_M(x)=v$. Furthermore, $x$ lies on the hypersurface $B_v^{-1}(c_v)$. By the definition of $c_v$, $B_v^{-1}(c_v)$ is a supporting hypersurface of $M$ at $x$ and hence $A-\nabla^2B_v\ge 0$ at $x$. The convexity of the Busemann function $B_v$ (cf. \cref{pro:RegularityBusemann} property~(3)) implies  $\nabla^2B_v\ge 0$ at $x$. Therefore, $A\ge A-\nabla^2B_v\ge 0$ at $x$. The estimate then follows from (1) and (2) of \cref{lem:CompareDeterminants}.
\end{proof}

\section{Proof of the main results}\label{sec:ProofMain}

\begin{proof}[Proof of \cref{thm:Main}]
    Since $S_M$ is a Lipschitz map (cf. \cref{thm:LipschitzGaussMap}) and surjective on $F$ (cf. \cref{rmk:PropertiesSetF}), the area formula for Lipschitz maps (cf. \cite[P47, (8.4)]{Sim83}) can be applied and it yields
    \begin{align*}
        \area(\mathbb{S}^n)=\area(S_M(F))\leq&\int_FJ_{S_M}(x)d\area_M.
    \end{align*}
    Observe that \cref{cor:JacobianEstimates} gives the estimate of the right hand side:
    \begin{align*}
        \int_FJ_{S_M}(x)d\area_M\leq e^{n(n+1)\kappa D}\int_F|GK(x)|d\area_M\leq e^{n(n+1)\kappa D}\int_M|GK(x)|d\area_M.
    \end{align*}
    The proof is completed.
\end{proof}

To prove \cref{cor:IsoperimetricIneq}, we follow the argument in \cite{Kle92}. Let $E_0\subseteq N$ be a compact domain with smooth boundary. We consider the isoperimetric profile
\begin{align}
I_{E_0}(V) = &\inf\{\operatorname{area}(\partial E): E\subseteq N \text{ is compact} \notag \\
&\quad \text{with smooth boundary, and } \operatorname{vol}(E)=V\}. 
\end{align}
The following theorem says that in low dimensions, $I_{E_0}(V)$ is attained at some smooth domain. 

\begin{theorem}[\text{\cite[Theorem 1.62]{Rit23}}]\label{lem:MinimizingProblem}
Let $N^{n+1}$ be a Riemannian manifold of dimension $n+1\leq7$, and $E_0\subseteq N^{n+1}$ a compact domain with smooth boundary. Then for any $V\in(0,\operatorname{vol}(E_0))$, there exists some $E\subseteq E_0$ with $C^{1,1}$ boundary such that $\operatorname{vol}(E)=V$ and that $\operatorname{area}(\partial E)=I_{E_0}(V)$. 
\end{theorem}

Let $E$ be a compact subset in $N$. We say that $S$ is a smooth supporting hypersurface for $E$ at $p\in E$ if $S$ is smooth, and $S\cap E=\{p\}$, and $E$ lies on the same side of $S$ as the inward normal vector at $p$. Denote by $\mathscr S(E,p)$ the collection of smooth supporting hypersurfaces for $E$ at $p$. Recall that the mean curvature measures the growth of area. Therefore we take
\begin{equation}
H_E = \sup \{H_S(p): p\in E, S\in\mathscr S(E,p)\}, 
\end{equation}
where $H_S(p)$ is the mean curvature of $S$ at $p$. This notion makes sense of mean curvature of a nonsmooth subset.

\begin{lemma}\label{lem:MeanCurvatureEstimate}
For any compact domain $E_0\subseteq N$ there holds
\[
H_{E}\geq e^{-2(n+1)\kappa\diam E} \overline H(\operatorname{area}(\partial E)), 
\]
where $\overline H(A)$ denotes the mean curvature of geodesic sphere in $\R^{n+1}$ with area $A$. 
\end{lemma}

\begin{proof}
Denote $E_0=E$ and let $D_0$ be the convex hull of $E_0$, and $D_s$ the $s$-neighborhood of $D_0$. Then $D_s$ is convex  and the boundary $C_s=\p D_s$ is $C^{1,1}$ for each $s>0$. Denote by $r_s:C_s\to C_0$ the nearest point retraction, and $\operatorname{Ricci}_-$ the infimum of Ricci curvature in $D_1$. Then as in \cite{Kle92}, for almost every $p\in r_s^{-1}(C_0\cap \partial E_0)$ we have 
\[
0 \leq H_{C_s}(p) \leq H_{E_0} - (\operatorname{Ricci}_-)s. 
\]
By regularization, \cref{thm:Main} holds for $C^{1,1}$ hypersurfaces; whence
\begin{align*}
&\quad e^{-n(n+1)\kappa\diam D_s} \operatorname{area}(\mathbb S^n) \leq \int_{C_s} GK_{C_s} \\
&= \int_{r_s^{-1}(C_0\cap\partial E_0)}GK_{C_s} + \int_{C_s\setminus r_s^{-1}(C_0\cap\partial E_0)}GK_{C_s} \\
&\leq \int_{r_s^{-1}(C_0\cap\partial E_0)}n^{-n}H_{C_s}^n + \int_{C_s\setminus r_s^{-1}(C_0\cap\partial E_0)}GK_{C_s} \\
&\leq \int_{r_s^{-1}(C_0\cap\partial E_0)}n^{-n}(H_{E_0}-(\operatorname{Ricci}_-)s)^n + \int_{C_s\setminus r_s^{-1}(C_0\cap\partial E_0)}GK_{C_s} \\
&\leq \operatorname{area}(r_s^{-1}(C_0\cap\partial E_0))n^{-n}(H_{E_0}-(\operatorname{Ricci}_-)s)^n + \int_{C_s\setminus r_s^{-1}(C_0\cap\partial E_0)}GK_{C_s}. 
\end{align*}
Sending $s\to0$, we have 
\[
\operatorname{area}(r_s^{-1}(\partial E_0))\to\operatorname{area}(C_0\cap\partial E_0) \quad \text{ and } \quad \int_{C_s\setminus r_s^{-1}(C_0\cap\partial E_0)}GK_{C_s} \to 0. 
\]
It follows that
\[
H_{E_0}\geq n e^{-(n+1)\kappa\diam D_0}\operatorname{area}(\mathbb S^n)^\frac{1}{n}\operatorname{area}(\partial E_0)^{-\frac{1}{n}} = e^{-(n+1)\kappa\diam D_0}\overline H(\operatorname{area}(\partial E_0)). 
\]
We complete the proof by noting that $\operatorname{diam}D_0\leq2\operatorname{diam}E_0$. 
\end{proof}

\begin{proof}[Proof of \cref{cor:IsoperimetricIneq}]
We fix any smooth domain $E_0\subseteq N$ with smooth boundary and take $V_0=\operatorname{vol}(E_0)$. Now for any $V\in(0,V_0)$, by \cref{lem:MinimizingProblem} there exists some $E\subseteq E_0$ with $C^{1,1}$ boundary such that
\[
\operatorname{area}(\partial E) = I_{E_0}(V) \quad \text{ and } \quad \operatorname{vol}(E) = V. 
\]
By \cref{lem:MeanCurvatureEstimate}, we derive
\begin{align*}
H_{E} &\geq e^{-2(n+1)\kappa\diam E} \overline H(\operatorname{area}(\partial E)) \\
&= ne^{-2(n+1)\kappa\diam E} \operatorname{area}(\mathbb S^n)^\frac{1}{n} \operatorname{area}(\partial E)^{-\frac{1}{n}} \\
&\geq ne^{-2(n+1)\kappa\diam E_0} \operatorname{area}(\mathbb S^n)^\frac{1}{n} I_{E_0}(V)^{-\frac{1}{n}}.
\end{align*}
By Lemma 10 in \cite{Kle92} (which holds in arbitrary dimension) we see that the left derivative of $I_{E_0}$ satisfies
\begin{align*}
D_-I_{E_0}(V) &\geq H_{E} \geq ne^{-2(n+1)\kappa\diam E_0} \operatorname{area}(\mathbb S^n)^\frac{1}{n} I_{E_0}(V)^{-\frac{1}{n}}
\end{align*}
for any $V\in(0,V_0)$. We integrate and deduce by continuity of $I_{E_0}$ that
\[
I_{E_0}(V)^\frac{n+1}{n} \geq (n+1)e^{-2(n+1)\kappa\diam E_0} \operatorname{area}(\mathbb S^n)^\frac{1}{n} V
\]
for any $V\in[0,V_0]$. In particular, we have isoperimetric inequality
\begin{align*}
\operatorname{area}(\partial E_0)^\frac{n+1}{n} &\geq I_{E_0}(V_0)^\frac{n+1}{n} \geq (n+1)e^{-2(n+1)\kappa\diam E_0} \operatorname{area}(\mathbb S^n)^\frac{1}{n} V_0 \\
&= (n+1)e^{-2(n+1)\kappa\diam E_0} \operatorname{area}(\mathbb S^n)^\frac{1}{n} \operatorname{vol}(E_0), 
\end{align*}
completing the proof. 
\end{proof}

\nocite{*}
\bibliographystyle{alpha_name-year-title}
\bibliography{References}

\begin{thebibliography}{BGS85}

\bibitem[Ara71]{Ara71}
Huzihiro Araki.
\newblock On quasifree states of {${\rm CAR}$} and {B}ogoliubov automorphisms.
\newblock {\em Publ. Res. Inst. Math. Sci.}, 6:385--442, 1970/71.

\bibitem[Aub76]{Aub76}
Thierry Aubin.
\newblock Probl\`emes isop\'erim\'etriques et espaces de {S}obolev.
\newblock {\em J. Differential Geometry}, 11(4):573--598, 1976.

\bibitem[BGS85]{BGS85}
Werner Ballmann, Mikhael Gromov, and Viktor Schroeder.
\newblock {\em Manifolds of nonpositive curvature}, volume~61 of {\em Progress in Mathematics}.
\newblock Birkh\"auser Boston, Inc., Boston, MA, 1985.

\bibitem[Bor01]{Bor01}
Albert Borb\'ely.
\newblock On the total curvature of hypersurfaces in negatively curved {R}iemannian manifolds.
\newblock {\em Publ. Math. Debrecen}, 59(1-2):17--23, 2001.

\bibitem[Bor02]{Bor02}
Albert Borb\'ely.
\newblock On the total curvature of convex hypersurfaces in hyperbolic spaces.
\newblock {\em Proc. Amer. Math. Soc.}, 130(3):849--854, 2002.

\bibitem[BH99]{BH99}
Martin~R. Bridson and Andr\'e Haefliger.
\newblock {\em Metric spaces of non-positive curvature}, volume 319 of {\em Grundlehren der mathematischen Wissenschaften [Fundamental Principles of Mathematical Sciences]}.
\newblock Springer-Verlag, Berlin, 1999.

\bibitem[BZ13]{BZ13}
Yurii~D Burago and Viktor~A Zalgaller.
\newblock {\em Geometric inequalities}, volume 285.
\newblock Springer Science \& Business Media, 2013.

\bibitem[CL57]{CL57}
Shiing-shen Chern and Richard~K. Lashof.
\newblock On the total curvature of immersed manifolds.
\newblock {\em Amer. J. Math.}, 79:306--318, 1957.

\bibitem[Cro80]{Cro80}
Christopher~B. Croke.
\newblock Some isoperimetric inequalities and eigenvalue estimates.
\newblock {\em Ann. Sci. \'Ecole Norm. Sup. (4)}, 13(4):419--435, 1980.

\bibitem[Dav23]{Dav23}
Colin Davalo.
\newblock Nearly geodesic immersions and domains of discontinuity, 2023.

\bibitem[Ebe96]{Ebe96}
Patrick~B. Eberlein.
\newblock {\em Geometry of nonpositively curved manifolds}.
\newblock Chicago Lectures in Mathematics. University of Chicago Press, Chicago, IL, 1996.

\bibitem[GS22]{GS22}
Mohammad Ghomi and Joel Spruck.
\newblock Total curvature and the isoperimetric inequality in {C}artan-{H}adamard manifolds.
\newblock {\em J. Geom. Anal.}, 32(2):Paper No. 50, 54, 2022.

\bibitem[Gro81]{Gro81}
Mikhael Gromov.
\newblock {\em Structures m\'etriques pour les vari\'et\'es riemanniennes}, volume~1 of {\em Textes Math\'ematiques [Mathematical Texts]}.
\newblock CEDIC, Paris, 1981.

\bibitem[HIH77]{HIH77}
Ernst Heintze and Hans-Christoph Im~Hof.
\newblock Geometry of horospheres.
\newblock {\em J. Differential Geometry}, 12(4):481--491, 1977.

\bibitem[Hel78]{Hel78}
Sigurdur Helgason.
\newblock {\em Differential geometry, {L}ie groups, and symmetric spaces}, volume~80 of {\em Pure and Applied Mathematics}.
\newblock Academic Press, Inc. [Harcourt Brace Jovanovich, Publishers], New York-London, 1978.

\bibitem[Kle92]{Kle92}
Bruce Kleiner.
\newblock An isoperimetric comparison theorem.
\newblock {\em Invent. Math.}, 108(1):37--47, 1992.

\bibitem[Kna02]{Kna02}
Anthony~W. Knapp.
\newblock {\em Lie groups beyond an introduction}, volume 140 of {\em Progress in Mathematics}.
\newblock Birkh\"auser Boston, Inc., Boston, MA, second edition, 2002.

\bibitem[MJ00]{MJ00}
Frank Morgan and David~L. Johnson.
\newblock Some sharp isoperimetric theorems for riemannian manifolds.
\newblock {\em Indiana Univ. Math. J.}, 49(2):1017--1041, 2000.

\bibitem[Rit23]{Rit23}
Manuel Ritor\'e.
\newblock {\em Isoperimetric inequalities in {R}iemannian manifolds}, volume 348 of {\em Progress in Mathematics}.
\newblock Birkh\"auser/Springer, Cham, [2023] \copyright 2023.

\bibitem[Sim83]{Sim83}
Leon Simon.
\newblock {\em Lectures on geometric measure theory}, volume~3 of {\em Proceedings of the Centre for Mathematical Analysis, Australian National University}.
\newblock Australian National University, Centre for Mathematical Analysis, Canberra, 1983.

\bibitem[Wei26]{Wei26}
Andr{\'e} Weil.
\newblock Sur les surfaces a courbure negative.
\newblock {\em CR Acad. Sci. Paris}, 182(2):1069--71, 1926.

\end{thebibliography}
\end{document}